\documentclass[12pt,a4paper,reqno]{amsart}
\allowdisplaybreaks
\usepackage{amsmath}
\usepackage{amsfonts}
\usepackage{amssymb,amsthm,amsfonts,amsthm,latexsym,enumerate,url,cases}
\numberwithin{equation}{section}
\usepackage{mathrsfs}
\usepackage{hyperref}
\hypersetup{colorlinks=true,citecolor=blue,linkcolor=blue,urlcolor=blue}
\def\pmod #1{\ ({\rm{mod}}\ #1)}

\theoremstyle{plain}
\newtheorem{theorem}{Theorem}
\newtheorem{lemma}{Lemma}

\newtheorem{proposition}{Proposition}

\theoremstyle{definition}

\usepackage{etoolbox}
\makeatletter
\patchcmd{\@settitle}{\uppercasenonmath\@title}{}{}{}
\patchcmd{\@setauthors}{\MakeUppercase}{}{}{}
\patchcmd{\section}{\scshape}{}{}{}
\makeatother

\begin{document}

\title
[{Solution to a problem of Luca, Menares and Pizarro-Madariaga}]
{Solution to a problem of Luca, Menares and Pizarro-Madariaga}

\author
[Y. Ding and L. Zhao] 
{Yuchen Ding \quad {\it and} \quad Lilu Zhao}

\address{(Yuchen Ding) School of Mathematical Sciences,  Yangzhou University, Yangzhou 225002, People's Republic of China}
\email{ycding@yzu.edu.cn}
\address{(Lilu Zhao) School of Mathematics, Shandong University, Jinan 250100, People's Republic of China}
\email{zhaolilu@sdu.edu.cn}


\begin{abstract}
Let $k\ge 2$ be a positive integer and $P^+(n)$ the greatest prime factor of a positive integer $n$ with convention $P^+(1)=1$. For any $\theta\in \left[\frac 1{2k},\frac{17}{32k}\right)$, set $$T_{k,\theta}(x)=\sum_{\substack{p_1\cdot\cdot\cdot p_k\le x\\ P^+(\gcd(p_1-1,...,p_k-1))\ge (p_1\cdot\cdot\cdot p_k)^\theta}}1,$$
where the $p'$s are primes. It is proved that 
$$T_{k,\theta}(x)\ll_{k}\frac{x^{1-\theta(k-1)}}{(\log x)^2},$$
which, together with the lower bound 
$$T_{k,\theta}(x)\gg_{k}\frac{x^{1-\theta(k-1)}}{(\log x)^2}$$
obtained by Wu in 2019, answer a 2015 problem of Luca, Menares and Pizarro-Madariaga on the exact order of magnitude of $T_{k,\theta}(x)$.

A main novelty in the proof is that, instead of using the Brun--Titchmarsh theorem to estimate the $k^{th}$ movement of primes in arithmetic progressions, we transform the movement to an estimation involving taking primes simultaneously by linear shifts of primes.
\end{abstract}
\maketitle

\section{Introduction}
Let $\mathcal{P}$ be the set of all primes and $p,q$ be primes throughout our paper.  For any $0<\theta<1$, let $$T_{\theta}(x)=\#\left\{p\le x: p\in \mathcal{P}, P^+(p-1)\ge p^\theta\right\},$$
where $P^+(n)$ denotes the largest prime factor of a positive integer $n$ with convention $P^+(1)=1$. In a cute article, Goldfeld \cite{Gold} investigated the large prime factors of shifted primes, where he showed that
$$T_{\frac{1}{2}}(x)\ge \frac{1}{2}\frac{x}{\log x}+O\left(x\log\log x/(\log x)^2\right)$$
as an elementary application of the Bombieri--Vinogradov theorem and the Brun--Titchmarsh inequality. Goldfeld further remarked that the same argument would lead to the fact
\begin{equation}\label{eq1-1}
\liminf_{x\rightarrow\infty}\frac{T_{\theta}(x)}{\pi(x)}>0
\end{equation}
providing that $\theta<\frac{7}{12}$, where $\pi(x)$ denotes the number of primes not exceeding $x$. Exploring large $\theta$ to satisfy Eq. (\ref{eq1-1}) is certainly a difficult and important research topic. Fouvry \cite{Fou} proved that there is some constant $\theta_0=0.6687$ such that (\ref{eq1-1}) holds. Fouvry's article is historically important for Fermat's Last Theorem as it provided the first evidence that the first case of Fermat's Last Theorem holds for infinitely many primes. For the connection between Fermat's Last Theorem and the shifted primes, one can refer to the work of Adleman and Heath-Brown \cite{AHB}. The best numerical value of $\theta$ satisfying Eq. (\ref{eq1-1}) is 0.677 \cite{B-H}, obtained by Baker and Harman. One can also consider shifted primes without large prime factors. For $\theta>0$, let
$$T^c_{\theta}(x)=\#\left\{p\le x: p\in \mathcal{P}, P^+(p-1)\le p^\theta\right\}.$$
Friedlander \cite{Friedlander} proved that 
\begin{equation*}
\liminf_{x\rightarrow\infty}\frac{T^c_{\theta}(x)}{\pi(x)}>0
\end{equation*}
is admissible for $\theta>0.303...$, improving the former results of Pomerance \cite{Po}, Balog \cite{Balog}, Fouvry and Grupp \cite{FoGr}. For $\theta\ge 0.2961$, Baker and Harman \cite{B-H} showed that there is some $a_1>0$ such that
$$T^c_{\theta}(x)\gg \frac{x}{(\log x)^{a_1}}.$$

It could be predicted that Eq. (\ref{eq1-1}) holds for any $\theta<1$. In fact, as far as 1980, Pomerance \cite{Po} had already conjectured that for any $\theta\in(0,1)$
$$T'_{\theta}(x):=\#\{p\le x: P^+(p-1)\ge x^\theta\}\sim \left(1-\rho\left(1/\theta\right)\right)\pi(x), \quad \text{as} \quad x\rightarrow \infty,$$
where $\rho(u)$ is the Dickman function, defined as the unique continuous solution of the equation
differential--difference
\begin{align*}
\begin{cases}
\rho(u)=1, & 0\le u\le 1,\\
u\rho'(u)=-\rho(u-1), & u>1.
\end{cases}
\end{align*}
Granville \cite{Gran} claimed that Pomerance's conjecture follows from the Eillott--Halberstam conjecture and later Wang \cite{Wang} gave an alternative proof of this claim. Motivated by a conjecture of Chen and Chen \cite{CC}, Wu \cite{Wu} proved that $$T_\theta(x)=T'_\theta(x)+O\left(x\log\log x/(\log x)^2\right),$$
which together with the above facts would imply that for any $\theta\in (0,1)$
\begin{equation}\label{eq1-2}
T_{\theta}(x)\sim \left(1-\rho\left(1/\theta\right)\right)\pi(x), \quad \text{as} \quad x\rightarrow \infty
\end{equation}
under assumption of the Eillott--Halberstam conjecture.
For early works related to large prime factors of shifted primes, one can also refer to the papers \cite{Iwa,Hoo1,Hoo2,Moto}.

Following this line, Luca, Menares and Pizarro-Madariaga \cite{LMP} considered the elaborate lower bound of $T_\theta(x)$ for small values $\theta$. They proved that for $\frac{1}{4}<\theta<\frac{1}{2}$
\begin{equation}\label{eq1-3}
T_\theta(x)\ge (1-\theta)\frac{x}{\log x}+E(x),
\end{equation}
with the error terms
$$
E(x)\ll
\begin{cases}
x\log\log x/(\log x)^2, & \text{for~} 1/4<c\le 1/2,\\
x/(\log x)^{5/3}, & \text{for~} c=1/4.
\end{cases}
$$
Pointing out by Chen and Chen \cite{CC}, the arguments given by Luca {\it et al.} (essentially due to Goldfeld) cannot lead to Eq. (\ref{eq1-3}) for $\theta\in \left(0,\frac{1}{4}\right)$. By some refinements of the method employed by Luca {\it et al.}, Chen and Chen extended the domain of $c$ to the interval $\left(0,\frac{1}{2}\right)$ with slightly better error terms of the order of the magnitude $O\left(x/(\log x)^2\right)$. Motivating by another conjecture of Chen and Chen \cite{CC}, the lower bound (\ref{eq1-3}) of Luca {\it et al.} was then improved to
$$T_{\theta}(x)\ge\left(1-4\int_{1/\theta-1}^{1/\theta}\frac{\rho(t)}{t}dt+o(1)\right)\pi(x)$$
by Feng and Wu \cite{FW}, provided $0<\theta<\theta_1$, where
$\theta_1\thickapprox 0.3517$ is the unique solution of equation
$$\theta-4\int_{1/\theta-1}^{1/\theta}\frac{\rho(t)}{t}dt=0.$$
Shortly after, for $0<\theta<\theta_2\thickapprox0.3734$ Liu, Wu and Xi \cite{LWX} improved the lower bound further to
$$T_{\theta}(x)\ge\left(1-4\rho(1/\theta)+o(1)\right)\pi(x),$$
where $\theta_2$ is the unique solution to equation $\theta-4\rho(1/\theta)=0$. If one believes the Eillott--Halberstam conjecture, then Eq. (\ref{eq1-2}) would imply that
$$\lim_{x\rightarrow \infty}T_\theta(x)/\pi(x)\rightarrow 0, \quad \text{as} \quad \theta\rightarrow 1,$$
which clearly means that there is some $\theta_0<1$ such that
\begin{equation}\label{eq1-4}
\limsup_{x\rightarrow \infty}T_{\theta_0}(x)/\pi(x)<\frac{1}{2}
\end{equation}
under assumption of the Eillott--Halberstam conjecture. Recently, the first author of the present article \cite{Ding} provided an unconditional proof of this matter (Eq. (\ref{eq1-4})), thus disproving a conjecture of Chen and Chen \cite{CC}.

In the same paper, Luca {\it et al.} \cite{LMP} investigated a variant on multiple variables of shifted primes whose motivation is Billerey and Menares' work \cite{BM} on the modularity of reducible mod $\ell$ Galois representations. Precisely, for any positive integer $k$ and $\theta\in \left(0,\frac{1}{k}\right)$, let
$$T_{k,\theta}(x)=\#\left\{p_1\cdot\cdot\cdot p_k\le x: P^+(\gcd(p_1-1,...,p_k-1))\ge (p_1\cdot\cdot\cdot p_k)^\theta\right\}.$$
They proved that for fixed $k\ge 2$ and $\theta\in \left[\frac{1}{2k},\frac{17}{32k}\right)$, there holds the following nontrivial bounds
\begin{equation}\label{eq1-5}
\frac{x^{1-\theta(k-1)}}{(\log x)^{k+1}}\ll_{k}T_{k,\theta}(x)\ll_{k}\frac{x^{1-\theta(k-1)}}{(\log x)^2}(\log\log x)^{k-1}.
\end{equation}
Luca {\it et al.} commented that `{\it Goldfeld's method does not seem to extend to this situation}', and they left as a problem for the readers to determine the exact order of magnitude of $T_{k,\theta}(x)$. With $k$ and $\theta$ in the same domains, Wu \cite{Wu} considerably improved their lower bound by showing that
\begin{equation}\label{eq1-5}
T_{k,\theta}(x)\gg_{k}\frac{x^{1-\theta(k-1)}}{(\log x)^{2}},
\end{equation}
which, as one can see, is close to the upper bound. Wu then commented that: {\it It seems reasonable to think that} $$T_{k,\theta}(x)\asymp_{k}\frac{x^{1-\theta(k-1)}}{(\log x)^{2}}(\log\log x)^{k-1}.$$
This is in fact not the case and we shall prove that the upper bound of $T_{k,\theta}(x)$ coincides with the lower bound given by Wu, thus the actual order of the magnitude of $T_{k,\theta}(x)$ is $$\frac{x^{1-\theta(k-1)}}{(\log x)^{2}}.$$
Let's restate it via the following theorem as an answer to the problem of Luca {\it et al.}
\begin{theorem}\label{thm1}
For any fixed $k\ge 2$ and $\theta\in \left[\frac{1}{2k},\frac{17}{32k}\right)$ we have
$$T_{k,\theta}(x)\asymp_{k}\frac{x^{1-\theta(k-1)}}{(\log x)^2}$$
providing that $x$ is sufficiently large, where the implied constants depend only on $k$.
\end{theorem}

\section{Proof of the upper bound}
The proof of the upper bound, among other things, is based on a well--known sieve result (weak form) (see e.g. {\cite[Theorem 2.4, Page 76]{Halberstam}}).
\begin{lemma}\label{lem3}
Let $g$ be a natural number, and let $a_i,b_i~(i=1,...,g)$ be integers satisfying
$$E:=\prod_{i=1}^ga_i\prod_{1\le r<s\le g}(a_rb_s-a_sb_r)\neq0.$$
Then
$$\#\{p\le y:a_ip+b_i\in \mathcal{P}, i=1,...,g\}\ll \prod_{p|E}\left(1-\frac1p\right)^{\rho(p)-g}\frac{y}{(\log y)^{g+1}},$$
where $\rho(p)$ denotes the number of solutions of
$$\prod_{i=1}^g(a_in+b_i)\equiv 0\pmod{p},$$
and where the constant implied by the $\ll-$symbol depends on $g$ only.
\end{lemma}

For any given positive integers $g$ and $\ell$, let
$$\mathcal{W}_{g,\ell}(z)=\sum_{1<h_1<\cdot\cdot\cdot<h_g<z}\frac{1}{h_1\cdot\cdot\cdot h_g}\prod_{p|E_{h_1,...,h_g}}\left(1+\frac{1}{p}\right)^\ell,$$ 
where $$E_{h_1,...,h_g}=h_1\cdot\cdot\cdot h_g\prod_{1\le i<j\le g}\left(h_j-h_i\right).$$

We need also a technical proposition below.

\begin{proposition}\label{pro1}
For any given positive integers $g$ and $\ell$, we have
$$\mathcal{W}_{g,\ell}(z)\ll (\log z)^g,$$
provided that $z$ is sufficiently large, where the implied constant depends only on $g$ and $\ell$.
\end{proposition}
\begin{proof}
It is easy to see that
$$\prod_{p|E_{h_1,...,h_g}}\left(1+\frac{1}{p}\right)^\ell\le \prod_{j=1}^{g}\prod_{p|h_j}\left(1+\frac{1}{p}\right)^\ell\prod_{1\le s<r\le g}\prod_{p|(h_r-h_s)}\left(1+\frac{1}{p}\right)^\ell.$$
Let $$A_j=\prod_{p|h_j}\left(1+\frac{1}{p}\right)^\ell \quad \text{and} \quad A_{r,s}=\prod_{p|(h_r-h_s)}\left(1+\frac{1}{p}\right)^\ell.$$
Employing the H\"older inequality, we obtain that
\begin{align}\label{equation0}
\mathcal{W}_{g,\ell}(z)&\le \sum_{1<h_1<\cdot\cdot\cdot<h_g<z}\frac{1}{h_1\cdot\cdot\cdot h_g}\prod_{j=1}^{g}A_j\prod_{1\le s<r\le g}A_{r,s}\nonumber\\
&\le\prod_{j=1}^{g}\Bigg(\sum_{\substack{1<h_j<z\\1\le j\le g}}\frac{1}{h_1\cdot\cdot\cdot h_g}A_j^G\Bigg)^{1/G}\prod_{1\le s<r\le g}\Bigg(\sum_{\substack{1<h_j<z\\1\le j\le g}}\frac{1}{h_1\cdot\cdot\cdot h_g}A_{r,s}^G\Bigg)^{1/G},
\end{align}
where $G=g+\binom{g}{2}=\binom{g+1}{2}$. We also write $L=2^{G\ell}$.
Let $\mu(d)$ be the M\"obius function and $\omega(d)$ the number of distinct prime factors of $d$. It is plain that
\begin{align*}
\sum_{1<h<z}\frac{1}{h}\prod_{p|h}\left(1+\frac{1}{p}\right)^{G\ell}
\le \sum_{1<h<z}\frac{1}{h}\prod_{p|h}\left(1+\frac{L}{p}\right)
=\sum_{1<h<z}\frac{1}{h}\sum_{d|h}\frac{\mu^2(d)L^{\omega(d)}}{d}.
\end{align*}
Exchanging the order of sums, we have
\begin{align*}
\sum_{1<h<z}\frac{1}{h}\sum_{d|h}\frac{\mu^2(d)L^{\omega(d)}}{d}&=\sum_{1\le d<z}\frac{\mu^2(d)L^{\omega(d)}}{d}\sum_{\substack{1<h<z\\ d|h}}\frac{1}{h}
\ll \log z\sum_{1\le d<z}\frac{\mu^2(d)L^{\omega(d)}}{d^2}\ll \log z
\end{align*}
due to the convergence of the series
$$\sum\limits_{1\le d<z}\frac{\mu^2(d)L^{\omega(d)}}{d^2},$$
from which we deduce that
\begin{align}\label{equation1}
\sum_{1<h<z}\frac{1}{h}\prod_{p|h}\left(1+\frac{1}{p}\right)^{G\ell}\ll \log z.
\end{align}
Now, for any $1\le j\le g$, by Eq. (\ref{equation1}) we have
\begin{align}\label{equation2}
\sum_{1<h_1<\cdot\cdot\cdot<h_g<z}\frac{1}{h_1\cdot\cdot\cdot h_g}A_j^G\ll (\log z)^{g-1}\sum_{1<h_j<z}\frac{1}{h_j}A_j^G\ll (\log z)^{g}.
\end{align}
For $1\le s<r\le g$, it can be noted that
\begin{align*}
\sum_{1<h_s<h_r<z}\frac{1}{h_sh_r}A_{r,s}^G&=\sum_{1<h_s<h_r<z}\frac{1}{h_sh_r}\prod_{p|(h_r-h_s)}\left(1+\frac{1}{p}\right)^{G\ell}\\
&\le \sum_{1<h_s<z}\sum_{1\le h<z}\frac{1}{h_s(h_s+h)}\prod_{p|h}\left(1+\frac{1}{p}\right)^{G\ell}\\
&<\log z\sum_{1\le h<z}\frac{1}{h}\prod_{p|h}\left(1+\frac{1}{p}\right)^{G\ell}\\
&\ll (\log z)^2,
\end{align*}
where the last estimate follows again from Eq. (\ref{equation1}).
It then can be concluded that
\begin{align}\label{equation3}
\sum_{1<h_1<\cdot\cdot\cdot<h_g<z}\frac{1}{h_1\cdot\cdot\cdot h_g}A_{r,s}^G\ll (\log z)^{g-2}\sum_{1<h_s<h_r<z}\frac{1}{h_sh_r}A_{r,s}^G\ll (\log z)^g.
\end{align}
The proposition now follows from Eqs. (\ref{equation0}), (\ref{equation2}) and (\ref{equation3}).
\end{proof}

It can be seen that trivial estimate would give the bound 
$$\ll (\log z)^g(\log\log z)^g$$
of the sum in Proposition \ref{pro1} for $\ell=g$ due to the well--known estimate
$$\prod_{p|z}\left(1+\frac{1}{p}\right)\ll \log\log z$$
(see e.g. \cite[Theorem 328]{Hardy}).
So, a $(\log\log z)^g$ factor is saved by Proposition \ref{pro1} which is crucial in our proof of the upper bound of main theorem.

The following lemma can be concluded from the paper of Luca {\it et al.}

\begin{lemma}\label{mainlemma}
For any fixed $k\ge 2$ and $\theta\in \left[\frac{1}{2k},\frac{17}{32k}\right)$ we have
$$T_{k,\theta}(x)\ll_k \frac{x}{\log x}\sum_{(x/2)^\theta<p\le x^{1/k}}\frac{1}{p}\Bigg(\sum_{\substack{q\le x\\ q\equiv 1\!\!\pmod{p}}}\frac{1}{q}\Bigg)^{k-1}+o\left(x^{1-\theta(k-1)}/(\log x)^2\right).$$
\end{lemma}
\begin{proof}
We know, from \cite[Eqs. (6) and (7)]{LMP}, that
\begin{align}\label{eq3-1}
T_{k,\theta}(x)&\ll_k \frac{x}{\log x}\sum_{(x/2)^\theta<p\le x^{1/k}}\frac{1}{p}\Bigg(\sum_{\substack{q\le x\\ q\equiv 1\!\!\pmod{p}}}\frac{1}{q}\Bigg)^{k-1}+x^{1/k+15(k-1)/(32k)}.
\end{align}
It is clear that
$$x^{1/k+15(k-1)/(32k)}=o\left(x^{1-\theta(k-1)}/(\log x)^2\right)$$
since $\theta\in\left[\frac{1}{2k},\frac{17}{32k}\right)$.
\end{proof}

To obtain the upper bound of $T_{k,\theta}(x)$, Luca {\it et al.} employed the Brun--Titchmarsh inequality (see e.g. \cite[Theorem 3.9]{MV}) and partial summation which led to the bound
$$\sum_{\substack{q\le x\\ q\equiv 1\!\!\pmod{p}}}\frac{1}{q}\ll \frac{\log\log x}{p}.$$
This would offer an extra factor $(\log\log x)^{k-1}$ in their final bound due to the quantity $\log\log x$ above. Note that even under the Eillott--Halberstam conjecture we cannot gain a correct size of the single prime counting function $\pi(y;p,1)$ when the modulus $p$ is close to the size like $y\exp(-\sqrt{\log y})$. Thus, the usual way for estimating single $\pi(y;p,1)$ and then summing over them is not applicable here. Our observation is that, ideally, the primes $q$ are evenly distributed among arithmetic progressions and thus we {\bf `should have'} 
$$\sum_{\substack{q\le x\\ q\equiv 1\!\!\pmod{p}}}\frac{1}{q}=\sum_{\substack{p<q\le x\\ q\equiv 1\!\!\pmod{p}}}\frac{1}{q}\overset{\text{heuristically}}{\ll}\frac{\log\log x-\log\log p}{p} \quad (\text{since}~~\varphi(p)=p-1),$$
from which we would get rid of a factor $(\log\log x)^{k-1}$ in the upper bound since 
$\log\log x-\log\log q$ is bound by a constant for $(x/2)^\theta<p\le x^{1/k}$. We intend to achieve this goal by rearranging the summations and then a sieve result (i.e., Lemma \ref{lem3}) could be used. We believe that the bound obtained by the sieve method reflects its real order of magnitude.

Now, let's fulfill our ideas.

\begin{proof}[Proof of the upper bound of Theorem \ref{thm1}]  Set
\begin{align}\label{eq3-2}
\mathcal{S}(x):=\sum_{(x/2)^\theta<p\le x^{1/k}}\frac{1}{p}\Bigg(\sum_{\substack{q\le x\\ q\equiv 1\!\!\pmod{p}}}\frac{1}{q}\Bigg)^{k-1}.
\end{align}
Opening up the sums in $S(x)$, we obtain that
\begin{align*}
\mathcal{S}(x)&=\sum_{(x/2)^\theta<p\le x^{1/k}}\frac{1}{p}\sum_{\substack{q_j\le x\\ q_j\equiv 1\!\!\pmod{p}\\1\le j\le k-1}}\frac{1}{q_1\cdot\cdot\cdot q_{k-1}}.
\end{align*}
Since $q_j\equiv 1\pmod{p}$, we can assume that $q_j=ph_j+1~(1\le j\le k-1)$. Then 
\begin{align}\label{eq3-3}
\mathcal{S}(x)\le \sum_{(x/2)^\theta<p\le x^{1/k}}\frac{1}{p^k}\sum_{\substack{1<h_j< x/p\\ph_j+1\in \mathcal{P} \\1\le j\le k-1}}\frac{1}{h_1\cdot\cdot\cdot h_{k-1}}.
\end{align}
Rearranging the sums in Eq. (\ref{eq3-3}), we have
\begin{align*}
\mathcal{S}(x)\le\sum_{\substack{1<h_1,...,h_{k-1}< 2^\theta x^{1-\theta}}}\frac{1}{h_1\cdot\cdot\cdot h_{k-1}}\sum_{\substack{(x/2)^\theta<p\le x^{1/k}\\ ph_j+1\in \mathcal{P} \\1\le j\le k-1}}\frac{1}{p^k}.
\end{align*}
By symmetries of the integers $h'$s above, the sum $\mathcal{S}(x)$ can be bounded as
\begin{align}\label{eq3-4}
\mathcal{S}(x)\ll_k\sum_{g=1}^{k-1}\sum_{\substack{1<h_1<\cdot\cdot\cdot<h_{g}< 2^\theta x^{1-\theta}}}\frac{1}{h_1\cdot\cdot\cdot h_{g}}\sum_{\substack{(x/2)^\theta<p\le x^{1/k}\\ ph_j+1\in \mathcal{P} \\1\le j\le g}}\frac{1}{p^k}.
\end{align}
The innermost sum above in Eq. (\ref{eq3-4}) can be bounded by Lemma \ref{lem3}. Actually, integrating by parts, we obtain that
\begin{align}\label{eq3-5}
\sum_{\substack{(x/2)^\theta<p\le x^{1/k}\\ ph_j+1\in \mathcal{P} \\1\le j\le g}}\frac{1}{p^k}=\frac{\mathcal{M}(x^{1/k})}{x}-\frac{\mathcal{M}((x/2)^{\theta})}{(x/2)^{k\theta}}+k\int\limits_{(x/2)^\theta}^{x^{1/k}}\frac{\mathcal{M}(t)}{t^{k+1}}dt,
\end{align}
where
$$\mathcal{M}(t)=\#\{p\le t:h_ip+1\in \mathcal{P}, i=1,...,g\}.$$
From Lemma \ref{lem3}, we know that
\begin{align}\label{eq3-6}
\mathcal{M}(t)\ll_k \prod_{p|E_{h_1,...,h_g}}\left(1-\frac1p\right)^{\rho(p)-g}\frac{t}{(\log t)^{g+1}}\ll \prod_{p|E_{h_1,...,h_g}}\left(1+\frac1p\right)^{g}\frac{t}{(\log t)^{g+1}},\end{align}
where $$E_{h_1,...,h_g}=h_1\cdot\cdot\cdot h_g\prod_{1\le i<j\le g}\left(h_j-h_i\right)\neq 0.$$
Taking Eq. (\ref{eq3-6}) into Eq. (\ref{eq3-5}), we will have
\begin{align}\label{eq3-7}
\sum_{\substack{(x/2)^\theta<p\le x^{1/k}\\ ph_j+1\in \mathcal{P} \\1\le j\le g}}\frac{1}{p^k}\ll_k \frac{x^{\theta(1-k)}}{(\log x)^{g+1}}\prod_{p|E_{h_1,...,h_g}}\left(1+\frac1p\right)^{g}.
\end{align}
Combing Eq. (\ref{eq3-7}) with Eq. (\ref{eq3-4}), we arrive at the bound
\begin{align*}
\mathcal{S}(x)\ll_k \sum_{g=1}^{k-1}\frac{x^{\theta(1-k)}}{(\log x)^{g+1}}\sum_{\substack{1<h_1<\cdot\cdot\cdot<h_{g}< 2^\theta x^{1-\theta}}}\frac{1}{h_1\cdot\cdot\cdot h_{g}}\prod_{p|E_{h_1,...,h_g}}\left(1+\frac1p\right)^{g}.
\end{align*}
It can be deduced from Proposition \ref{pro1} with $\ell=g$ that
\begin{align*}
\sum_{\substack{1<h_1<\cdot\cdot\cdot<h_{g}< 2^\theta x^{1-\theta}}}\frac{1}{h_1\cdot\cdot\cdot h_{g}}\prod_{p|E_{h_1,...,h_g}}\left(1+\frac1p\right)^{g}\ll_g (\log 2^\theta x^{1-\theta})^g\ll_g (\log x)^g,
\end{align*}
which clearly implies that
\begin{align}\label{eq3-8}
\mathcal{S}(x)\ll_k \sum_{g=1}^{k-1}\frac{x^{\theta(1-k)}}{\log x}\ll _k \frac{x^{\theta(1-k)}}{\log x}.
\end{align}
Our theorem now follows from Lemma \ref{mainlemma} and Eqs. (\ref{eq3-2}), (\ref{eq3-8}).
\end{proof}

\section*{Acknowledgments}
The first author is supported by National Natural Science Foundation of China  (Grant No. 12201544), Natural Science Foundation of Jiangsu Province, China (Grant No. BK20210784), China Postdoctoral Science Foundation (Grant No. 2022M710121), the foundations of the projects "Jiangsu Provincial Double--Innovation Doctor Program'' (Grant No. JSSCBS20211023) and "Golden  Phoenix of the Green City--Yang Zhou'' to excellent PhD (Grant No. YZLYJF2020PHD051).

The second author is support by the National Key Research and Development Program of China (Grant No. 2021YFA1000700) and National Natural Science Foundation of China (Grant No. 11922113).

\end{document}